\documentclass[11pt,oneside]{amsart}
\usepackage{amsfonts}
\usepackage{amsmath}
\usepackage{amssymb}
\usepackage{multirow}
\usepackage{lscape}
\input xy \xyoption{all}

\usepackage{rotating}
\usepackage{hyperref}
\usepackage{color}

\newtheorem{theorem}{Theorem}[section]

\newtheorem{definition}[theorem]{Definition}
\newtheorem{example}[theorem]{Example}

\newtheorem{lemma}[theorem]{Lemma}

\newtheorem{proposition}[theorem]{Proposition}

\newtheorem{remark}[theorem]{Remark}

\newcommand{\f}{\mathfrak}
\newcommand{\rarrow}{\rightarrow}

\newcommand{\bb}{\mathbb}
\newcommand{\R}{\mathbb{R}}

\newcommand{\e}{\varepsilon}

\title[Torsion-free $G_{2(2)}^*$-structures with full holonomy]{Torsion-free $G_{2(2)}^*$-structures with full holonomy on nilmanifolds}

\begin{document}

\title[Torsion-free $G_{2(2)}^*$-structures with full holonomy]{Torsion-free $G_{2(2)}^*$-structures with full holonomy on nilmanifolds}

\author[A. Fino, I.  Luj\'{a}n]{Anna Fino and Ignacio
Luj\'{a}n}
\address{Dipartimento di Matematica G. Peano \\ Universit\`a di Torino\\
Via Carlo Alberto 10\\
10123 Torino\\ Italy} \email{annamaria.fino@unito.it}
\address{Departamento de Geometr\'{i}a y Topolog\'{i}a\\
Facultad de Matem\'{a}ticas, Universidad Complutense de Madrid\\
28040 Madrid, Spain}
  \email{ilujan@mat.ucm.es}
\subjclass{Primary 53C10. Secondary 53C25, 53C29, 53C30, 53C50}

\thanks{The first author was  partially supported by GNSAGA (Indam) of Italy. The second author was supported by a FPU grant from the Spanish Ministerio de Educaci\'on, and partially supported by project MTM2011-
22528 from MINECO (Spain).}

%

\begin{abstract}
We study  the existence of  invariant   metrics with   holonomy
$G_{2(2)}^* \subset SO(4,3)$ on compact nilmanifolds, i.e. on
compact quotients of nilpotent  Lie groups by discrete subgroups.
We prove  that, up to isomorphism,  there exists only one
indecomposable nilpotent Lie algebra  admitting  a   torsion-free
$G_{2(2)}^*$-structure  such that the center is definite    with
respect to the induced inner product.   In particular,  we show
that the associated  compact nilmanifold  admits  a $3$-parameter
family of  invariant metrics with full holonomy $G_{2(2)}^*$.
\end{abstract}

\maketitle

\section{Introduction}

The holonomy group of a pseudo-Riemannian manifold $(M, g)$ at a
point  $p \in M$ is defined as the group of parallel transports
along loops based at $p$.  In \cite{Berger} Berger gave a list of
possible holonomy groups of simply connected (pseudo-)Riemannian
manifolds under the assumption that the group acts irreducibly on
the tangent space at $p$. In the list,  the exceptional compact
Lie group $G_2$ appears as the holonomy group of a $7$-dimensional
Riemannian manifold, and its non-compact real form $G_{2(2)}^*
\subset SO(4,3)$, as the holonomy group of a manifold with metric
of signature $(4, 3)$. Bryant proved in \cite{Bryant} that there
exist pseudo-Riemannian metrics with exceptional holonomy groups,
in particular, with holonomy $G_2$ and $G_{2(2)}^*$. Few explicit
examples of signature (4, 3)-metrics with holonomy group
$G_{2(2)}^*$ have been constructed explicitly, see for instance
\cite{CCLS},  \cite{GW} and \cite{Leistner-Nurowski}.

Each Riemannian manifold whose holonomy group is contained in
$G_2$ is Ricci-flat. In particular, if the holonomy group of a
homogeneus space $M$  is contained in $G_2$, then the homogeneous
metric has to be flat. In contrast, in \cite{Kath2} it has been
shown that there exist indecomposable indefinite symmetric spaces
of signature $(4,3)$ whose holonomy is contained in $G_{2(2)}^*$.

In the present paper we  show that there exist compact
nilmanifolds, i.e.  compact quotients of  simply-connected
nilpotent Lie groups  $G$ by  uniform   discrete subgroups
$\Gamma$,  admitting  invariant metrics of signature $(4, 3)$ with
holonomy equal to $G_{2(2)}^*$.  By invariant metric on $G/\Gamma$
we mean  a metric induced by a inner product on the Lie algebra of
$G$. More precisely, we prove that, up to isomorphism, there
exists only one indecomposable $7$-dimensional nilpotent Lie
algebra admitting a torsion-free $G_{2(2)}^*$-structure $\varphi$
such that the center is definite with respect to the induced inner
product $g_{\varphi}$. The Lie algebra has  structure equations
$$\begin{array}{l}
[e_1, e_2] = - e_4, \,  [e_2, e_3] = - e_5, \,  [e_1, e_3] = e_6,\\[2pt]
[e_2, e_6] = - [e_3, e_4] = - [e_1, e_6] = [e_2, e_5] = - 2e_7.
\end{array}$$
For this Lie algebra we exhibit a $3$-parameter family of
(non-symmetric) metrics with full holonomy $G_{2(2)}^*$. This Lie
algebra gives rise to a compact nilmanifold  which inherits a
$3$-parameter family with holonomy
equal to $G_{2(2)}^*$.  

By Nomizu's theorem  \cite{Nomizu} the de Rham cohomology of a
compact nilmanifold  $G/\Gamma$ can be calculated using invariant
differential forms and is isomorphic to the  Chevalley-Eilenberg
cohomology  of the Lie algebra $\frak g$ of $G$. Moreover,  a
compact nilmanifold  is formal  in the sense of Sullivan's
rational homotopy theory \cite{Sullivan} if and only if it is a
torus. Therefore the compact examples  with full holonomy
$G_{2(2)}^*$ that we get  are not formal  in the sense of
Sullivan's rational homotopy theory \cite{Sullivan}. It is still
an open problem to see  if  compact Riemannian manifolds with
holonomy $G_2$ are formal.

Other examples of torsion-free $G^*_{2(2)}$-structures with
$1$-dimensional and $2$-dimensional holonomy have been found by M.
Freibert on almost Abelian Lie algebras \cite{Freibert}, showing
that there are  examples of calibrated Ricci-flat
$G^*_{2(2)}$-structures on  compact nilmanifolds which are not
parallel and do not have holonomy contained in $G^*_{2(2)}$. In
addition, it is worth mentioning the recent work \cite{Willse}
where manifolds with $5$-dimensional holonomy contained in
$G^*_{2(2)}$ are constructed.

\smallskip

{\it Acknowledgements.} We would like to thank Marco Freibert for
pointing out a mistake in the previous version of the paper and
for useful comments. In addition we would like to thank the
referee for valuable comments and corrections.

\section{Preliminaries}

For more details on the group $G_{2(2)}^*$ and
$G_{2(2)}^*$-structures as for the proofs not appearing in this
section see \cite{Kath}.

\smallskip

Let $M$ be a $7$-dimensional manifold, and $L(M)$ its bundle of
linear frames.

\begin{definition}
A $G^*_{2(2)}$-structure on $M$ is a $G^*_{2(2)}$-reduction of
$L(M)$.
\end{definition}

A $G^*_{2(2)}$-structure   $P\to M$   on $M$ determines a global
$3$-form $\varphi$ defined by the equivariant map
$$
\varphi: P  \rarrow  \bigwedge^3\bb{R}^7, \quad
              u  \mapsto  \varphi_0,
$$
with
$$\varphi_0=-e^{127}-e^{135}+e^{146}+e^{236}+e^{245}-e^{347}+e^{567},$$
where  $e^{ijk}$ stands for $e^i\wedge e^j\wedge e^k$. Conversely,
 a $3$-form $\varphi$,  $i.e.$ a section of
$\bigwedge^3_*(M)=L(M)\times_{GL(7)}\bigwedge^3_*(\R^7)$, where
$\bigwedge^3_*(\R^7)$ is the $GL(7)$-orbit of $\varphi_0$,
determines a $G_{2(2)}^*$-structure
$$P=\{u\in L(M) \mid u^*\varphi_0=\varphi\}.$$
The inclusion $G_{2(2)}^*\subset SO^+(4,3)$ induces  a
pseudo-Riemannian metric $g$ of index $4$ and a space and time
orientation on $M$. For the proof of the following proposition see
for instance \cite{Bryant} and \cite{Gray}.

\begin{proposition}
Let $P$ be a $G_{2(2)}^*$-structure on $M$, and $\nabla$ the
Levi-Civita connection of the associated  metric $g$.  The
following  conditions   are equivalent:
\begin{enumerate}
\item[(a)]  $\nabla \varphi=0$;

\item[(b)] $d\varphi=0$ and $d(*\varphi)=0$, where $*$ is the
Hodge star operator of $g$;

\item[(c)] The holonomy group of $g$ is isomorphic to a subgroup
of $G_{2(2)}^*$.


\end{enumerate}
\end{proposition}
The $G_{2(2)}^*$-structure $P$ is  called {\em torsion-free} if
one of the previous equivalent conditions hold.

Let $M$ be a manifold  endowed with a $G_{2(2)}^*$-structure
determined by a $G_{2(2)}^*$-reduction $P$ of $L(M)$ and
associated $3$-form $\varphi$. Suppose that there is a Lie group
$G$ acting transitively on $M$. We say that the
$G_{2(2)}^*$-structure is $G$-invariant if $P$ is invariant by the
lifted action of $G$ on $L(M)$. By the definition of $\varphi$,
this is equivalent to $\varphi$ being $G$-invariant. In addition,
the inclusion of $P$ in the frame of orthonormal frames with
respect to the associated metric $g_{\varphi}$ implies that $G$
acts on $M$ by isometries with respect to this metric.

\subsection{Almost special $\varepsilon$-Hermitian structures in six
dimensions}\label{special hermitian section}

We recall that a $k$-form on a  $n$-dimensional real vector space
$V$   is called {\em stable} if its $GL(V)$ orbit is an open
subset of $\bigwedge^k V^*$. For a collection of basic facts about
stable forms and the proofs of the forthcoming  results  in this
section see \cite{CCLS}.

Regarding $2$-forms, we  recall that a $2$-form $\omega$ on a real
vector space $V$ of dimension $n=2m$ is stable if and only if it
is non-degenerate, that is, $\omega^m\neq 0$. Let now $V$ be an
oriented $6$-dimensional vector space. We consider the canonical
isomorphism
$$
k:  \bigwedge^5 V^*  \rarrow  V\otimes \bigwedge^6 V^*, \quad
     \xi         \mapsto   X\otimes\nu
     $$
     with $i_X\nu=\xi.$

Let $\rho$ be a $3$-form on $V$, we define (see \cite{CCLS})
\begin{equation}\label{K sub rho stable}
K_{\rho}(v)=k((i_v\rho)\wedge \rho)\in V\otimes \bigwedge^6V^*
\end{equation}
\begin{equation}\label{lambda stable}
\lambda(\rho)=\frac{1}{6}\mathrm{tr}(K_{\rho}^2)\in
\left(\bigwedge^6 V^* \right)^{\otimes 2}
\end{equation}
\begin{equation}\label{phi stable}
\phi(\rho)=\sqrt{|\lambda(\rho)|}\in\bigwedge^6 V^*
\end{equation}
where we have chosen the positive square root with respect to the
orientation of $V$. In the case $\lambda(\rho)\neq 0$ we also
define
\begin{equation}\label{J rho stable}
J_{\rho}=\frac{1}{\phi(\rho)}K_{\rho}\in \mathrm{End}(V).
\end{equation}

\begin{proposition}[\cite{CCLS}]
A $3$-form $\rho$ on $V$ is stable if and only if
$\lambda(\rho)\neq 0$. In that case there are two orbits
corresponding to $\lambda(\rho)>0$ and $\lambda(\rho)<0$.
\end{proposition}
The elements in the orbit corresponding to
$\lambda(\rho)>0$ have stabilizer $SL(3,\R)\times SL(3,\R)$ in
$GL^+(V)$ and $J_{\rho}$ is an almost para-complex structure,
i.e., $J_{\rho}^2=1$, $J_{\rho}\neq 1$. The elements in the orbit
corresponding to $\lambda(\rho)<0$ have stabilizer $SL(3,\bb{C})$
in $GL^+(V)$ and $J_{\rho}$ is an almost complex structure, i.e,
$J_{\rho}^2=-1$.  In both cases the dual form of $\rho$ is
determined by the formula
$$\hat{\rho}=J_{\rho}^*\rho.$$
Finally, it is easy to prove that $\hat{\hat{\rho}}=-\rho$ and
\begin{equation}\label{estructuras complejas}
J_{\hat{\rho}}=-\epsilon J_{\rho}.
\end{equation}

A pair $(\omega,\rho)\in\bigwedge^2V^*\times\bigwedge^3V^*$ of
stable forms is called \textit{compatible} if $\omega\wedge\rho=0$
(or equivalently $\hat{\rho}\wedge\omega=0$) and
\textit{normalized} if $\hat{\rho}\wedge\rho=\frac{2}{3}\omega^3$.
Let $\varepsilon$ be the sign of $\lambda(\rho)$, every compatible
pair $(\omega,\rho)$ uniquely determines an $\varepsilon$-complex
structure $J_{\rho}$ (i.e., $J_{\rho}^2=\varepsilon$), an inner
product $g_{(\omega,\rho)}=\varepsilon\omega(\cdot,J_{\rho}\cdot)$
(of signature $(3,3)$ for $\varepsilon=1$, and
definite  or of signature $(2,4)$ or $(4,2)$ for
$\varepsilon=-1$), and a complex volume form
$\Psi=\rho+i_{\varepsilon}\hat{\rho}$ of type $(3,0)$ with respect
to $J_{\rho}$ (where $i_{\varepsilon}$ is the complex or
para-complex imaginary unit accordingly). In addition, the
stabilizer of the pair $(\omega,\rho)$ under $GL(V)$ is $SU(p,q)$
for $\varepsilon=-1$ and $SL(3,\R)\subset SO(3,3)$ for
$\varepsilon=1$.

\section{Torsion-free $G_{2(2)}^*$-structures on  nilpotent Lie
algebras}\label{section torsion free G2(2)-structures on nilpotent
Lie algebras}

Let $\f{g}$ be a $7$-dimensional nilpotent Lie algebra with a
three form  $\varphi \in \Lambda^3 \f{g}^*$ defining a
$G^*_{2(2)}$-structure. Let $\xi$ be an element in the center of
$\f{g}$ such that $g_{\varphi}(\xi,\xi)\neq 0$, where
$g_{\varphi}$ is the inner product of signature $(4,3)$ induced by
$\varphi$.  The quotient $\f{h}=\f{g}/\mathrm{span}\{\xi\}$  has a
unique Lie algebra structure making $\f{h}$ nilpotent and the
projection map $\f{g}\to\f{h}$ is  a Lie algebra epimorphism. Via
the pullback we identify basic forms on $\f{g}$ (i.e., forms
$\alpha$ with $i_{\xi}\alpha=0$) with forms on $\f{h}$. Assume
that $g_{\varphi}(\xi,\xi)=-\varepsilon\in\{\pm 1\}$. Let
$\eta=-\varepsilon\xi^{\flat}$ so that $\eta(\xi)=1$, in analogy
with circle bundles we can think of $\eta$ as a connection form on
the bundle $\f{g}\to\f{h}$, and $d\eta$ as its curvature form. We
write
\begin{equation}\label{varphi}\varphi=\omega\wedge\eta+\psi_+,\end{equation} where
$\omega=i_{\xi}\varphi$ and $\psi_+$ are basic. A simple
computation shows that
\begin{equation}\label{starvarphi}*\varphi=\varepsilon\psi_-\wedge\eta-\frac{1}{2}\varepsilon\omega\wedge\omega,\end{equation}
where $\psi_-=\varepsilon i_{\xi}(*\varphi)=\hat{\psi}_+$ is
basic. Since $g_{\varphi}(\xi,\xi)=\pm 1$, by the stability of
$\varphi$, there is always a   basis $\{E_1,\ldots,E_7\}$ of
$\frak g$ such that
$$\varphi=\varepsilon (E^{127}+E^{347})+E^{567}+E^{135}+\varepsilon(E^{146}+E^{236}+E^{245})$$
and $\xi=E_7$. Such basis is orthonormal, so that $\eta=E^7$ and
we obtain $$\omega=\varepsilon (E^{12}+E^{34})+E^{56}, \quad
\psi_+=E^{135}+\varepsilon(E^{146}+E^{236}+E^{245}).$$ A simple
computation using the Hodge dual of $\varphi$ shows that with
respect to this basis $\psi_-=
E^{246}+\varepsilon(E^{235}+E^{145}+E^{136})$. Therefore it is
easy to see that $\omega$ and $\psi_+$ are stable and compatible,
hence $(\omega,\psi_+)$ defines a special $\varepsilon$-Hermitian
structure on $\f{h}$. Moreover, the pair $(\omega,\psi_+)$ is
normalized as $\varepsilon\in\{\pm 1\}$. The same is true for
$(\omega,\psi_-)$ (see Proposition 1.14 in \cite{CCLS}).

We now suppose that $\varphi$ is torsion free, i.e., $d\varphi=0$,
$d(*\varphi)=0$. Since $\xi$ is in the center of $\f{g}$ one has
that $\mathcal{L}_{\xi}\varphi=0$ and
$\mathcal{L}_{\xi}(*\varphi)=0$, so that
$$d\omega=d(i_{\xi}\varphi)=\mathcal{L}_{\xi}\varphi=0,$$
and
$$d\psi_-=\varepsilon d(i_{\xi}(*\varphi))=\varepsilon \mathcal{L}_{\xi}(*\varphi)=0.$$
This means that the pair $(\omega,\psi_-)$ is a \textit{symplectic
half-flat structure} on $\f{h}$. Finally taking exterior
derivative in \eqref{varphi} and \eqref{starvarphi} we obtain
$$0=d\varphi=\omega\wedge d\eta+d\psi_+,\qquad 0=d(*\varphi)=-\varepsilon\psi_-\wedge d\eta.$$
Therefore, constructing torsion-free $G^*_{2(2)}$-structures on
$\f{g}$ with $g_{\varphi}(\xi,\xi)=-\varepsilon\in\{\pm 1\}$ is
equivalent to construct symplectic half-flat structures on $\f{h}$
satisfying
\begin{equation}\label{additional equations}\omega\wedge d\eta+d\psi_+=0,\qquad \psi_-\wedge
d\eta=0.\end{equation}

\begin{remark}
With the previous procedure we can obtain all the torsion-free
$G_{2(2)}^*$-structu\-res for which $\xi$ is unitary. However, let
$(\omega,\psi_+)$ be normalized and let $t>0$. The $3$-form
$\varphi_t=t^{\frac{1}{2}}\omega\wedge\eta+\psi_+$ is a
$G_{2(2)}^*$-structure with associated inner product $g_t$
satisfying $g_t(\xi,\xi)=-\varepsilon t$.
\end{remark}


We will restrict ourselves to indecomposable $7$-dimensional
nilpotent Lie algebras. In order to apply the reduction procedure
it is reasonable first  to consider $G^*_{2(2)}$-structures for
which the center is definite with respect to $g_{\varphi}$.
Although as we shall see this condition is rather strong, it will
lead us to examples of metrics with full holonomy $G_{2(2)}^*$.

Following \cite{CF} we have the following obstructions to the
existence of a parallel $G^*_{2(2)}$-structure on $\f{g}$.

\begin{proposition}\label{obstruction 1}
Let $\f{g}$ be a $7$-dimensional Lie algebra admitting a
$G^*_{2(2)}$-structure with associated $3$-form $\varphi$. Then
$(i_X\varphi)^3\neq 0$ for every element $X\in \f{g}$ with
$g_{\varphi}(X,X)\neq 0$.
\end{proposition}

\begin{proof}
Let $X$ be an element of $\f{g}$ with $g_{\varphi}(X,X)\neq 0$,
then without loss of generality we can assume that
$g_{\varphi}(X,X)=-\varepsilon\in\{\pm 1\}$. Let
$\eta=-\varepsilon X^{\flat}$, we write
$\varphi=\omega\wedge\eta+\psi_+$. As seen before there is a basis
$\{E_1,\ldots,E_7\}$ of $\f{g}$ such that
$$\varphi=\varepsilon (E^{127}+E^{347})+E^{567}+E^{135}+\varepsilon(E^{146}+E^{236}+E^{245}),$$
with $E_7=X$. Therefore
$$i_{X}\varphi=\varepsilon (E^{12}+E^{34})+E^{56},$$
so that $(i_X\varphi)^3\neq 0$.
\end{proof}

Making use of Proposition \ref{obstruction 1}, we can see that a
Lie algebra $\f{g}$ does not admit a calibrated
$G^*_{2(2)}$-structure, i.e. with  associated $3$-form $\varphi$
satisfying $d\varphi=0$, with definite center by finding an
element $\xi$ in the center of $\f{g}$ such that $\left(i_
\xi\varphi\right)^3=0$ for every closed $3$-form
$\varphi\in\bigwedge^3\f{g}^*$. Note that this is an obstruction
to the existence of a $G^*_{2(2)}$-structure with definite center
in general, without any assumption about $d\varphi$ or
$d(*\varphi)$.

\begin{proposition}\label{obstruction 2}
Let $\f{g}$ be a  $7$-dimensional nilpotent Lie algebra with a
calibrated $G^*_{2(2)}$-structure $\varphi$ with definite center.
If $\pi:\f{g}\to \f{h}$ is a Lie algebra epimorphism with kernel
contained in the center and $\f{h}$ is $6$-dimensional, then
$\f{h}$ admits a symplectic form $\omega$ and the curvature form
is in the kernel of
$$\cdot\wedge\omega: H^2(\f{h}^*)\to H^4(\f{h}^*).$$
Moreover, if the curvature form is exact, then $\f{g}$ is
isomorphic to $\R\oplus\f{h}$ as Lie algebras.
\end{proposition}

\begin{proof}
Let $\xi$ be in the kernel of $\pi$, as $\xi$ is in the center of
$\f{g}$ we have that $g_{\varphi}(\xi,\xi)\neq 0$, and we can
suppose that $g_{\varphi}(\xi,\xi)=-\varepsilon\in\{\pm 1\}$. Let
$\eta=-\varepsilon\xi^{\flat}$, we write $\varphi=
\pi^*\omega\wedge\eta+ \pi^*\psi_+$, where $\omega$ and $\psi_+$
are forms on $\f{h}$. Since $\pi^*\omega=i_{\xi}\varphi$ we have
that $\omega$ is non-degenerate. Moreover,
$$0=d\varphi=d\pi^*\omega\wedge\eta+\pi^*\omega\wedge d\eta+d\pi^*\psi_+,$$
where $d\pi^*\omega$, $d\eta$ and $d\pi^*\psi_+$ are basic.
Therefore $\omega$ is a symplectic form and $d\eta$ is in the
kernel of $\cdot\wedge\omega: H^2(\f{h}^*)\to H^4(\f{h}^*)$. If
the form $d\eta$ is exact on $\frak h$ then the epimorphism
$\f{g}\to\f{h}$ is trivial, hence $\f{g}$ is isomorphic to
$\R\oplus\f{h}$ as Lie algebras.
\end{proof}

Since a torsion-free $G^*_{2(2)}$-structure is in particular
calibrated, making use of Proposition \ref{obstruction 2} we can
show that a Lie algebra $\f{g}$ does not admit a torsion-free
$G^*_{2(2)}$-structure with definite center by finding an element
$\xi$ in the center of $\f{g}$ contradicting Proposition
\ref{obstruction 2}, i.e.,  such that
$\f{h}=\f{g}/\mathrm{Span}\{\xi\}$ does not admit a symplectic
form $\omega$, or such that $d\eta$ is not in the kernel of
$\cdot\wedge\omega: H^2(\f{h}^*)\to H^4(\f{h}^*)$.

The aim now is to take Gong's classification of indecomposable
$7$-dimen\-sional Lie algebras (see \cite{Gong}), and eliminate
those Lie algebras $\f{g}$ for which there is an element
$\xi\in\f{g}$ contradicting Proposition \ref{obstruction 1} or
Proposition \ref{obstruction 2}, because as we have seen those Lie
algebras cannot admit a torsion-free $G^*_{2(2)}$-structure with
definite center. In order to do that we use the work done in
\cite{CF}, where analogous obstructions are used to classify
$7$-dimensional nilpotent Lie algebras admitting a calibrated
$G_2$-structure. More precisely, in that paper the authors start
with Gong's list and eliminate those Lie algebras $\f{g}$ which do
not admit a calibrated $G_2$-structure by finding an element
$\xi\in\f{g}$  contradicting Proposition 1 or Lemma 3 therein,
which are the analogues of Proposition \ref{obstruction 1} and
Proposition \ref{obstruction 2} for the Riemannian case. This is
done in the proof of Lemma 5 and in the Appendix of \cite{CF}. In
our case, since the obstructions given by Proposition
\ref{obstruction 1} and Proposition \ref{obstruction 2} (as well
as the obstructions given by Proposition 1 or Lemma 3 of
\cite{CF}) only depend on the structure of the space of $3$-forms
and the space of closed $2$-forms (and not in the signature of the
metric), a simple inspection shows that for every Lie algebra in
Gong's list eliminated in \cite{CF},  the same element $\xi$ used
in \cite{CF} also contradicts Proposition \ref{obstruction 1} or
Proposition \ref{obstruction 2}. This means that any of the Lie
algebras eliminated in \cite{CF} can admit a torsion-free
$G^*_{2(2)}$-structure. The only indecomposable Lie algebras in
Gong's list not yielding a contradiction with Proposition
\ref{obstruction 1} and Proposition \ref{obstruction 2} are
\begin{equation}\label{list1}
\begin{array}{l}
(0,0,12,0,0,13+24,15),\\
(0,0,12,0,24+13,14,46+34+15+23), \,(0,0,12,0,0,13,14+25),\\
(0,0,12,0,13,24+23,25+34+16+15-3 \cdot {26}),\\
(0,0,0,12,23,-13,2 \cdot {26}-2 \cdot{34}-2 \cdot{16}+ 2 \cdot  {25}),\\
(0,0,0,12,13,14+23,15),  \, (0,0,0,12,13,14,15)\\
(0,0,12,13,23,15+24,16+34),\\
(0,0,12,13,23,15+24,16+25+34).
\end{array} \end{equation}
The notation $(0,0,0,12,23,-13,2 \cdot {26}-2 \cdot{34}-2
\cdot{16}+ 2 \cdot {25})$  means that $\f{g}^*$ has a basis
$\{e^1,\ldots,e^7\}$ such that $de^i=0$, for $i=1,\ldots,3$ and
 $de^4=e^{12}$,  $d e^5= e^{23}$, $d e^6 = - e^{13}$, $de^7 = 2 e^{26} - 2 e^{34} - 2 e^{16} + 2
 e^{25}$. We now show which of these Lie algebras actually admit a
 torsion-free $G^*_{2(2)}$-structure with definite center.

\begin{lemma}\label{lemma eliminating}
With the exception of $$\f{g}_1=(0,0,0,12,23,-13,2 \cdot {26}-2
\cdot{34}-2 \cdot{16}+ 2 \cdot {25}),$$ none of
the Lie algebras in \eqref{list1} admits a torsion-free
$G^*_{2(2)}$-structure with definite center.
\end{lemma}

\begin{proof}
We have seen that if a Lie algebra $\f{g}$ admits a torsion-free
$G^*_{2(2)}$-structure with definite center, then for every
element $\xi$ in the center of $\f{g}$ we have a reduction
$\f{g}\to\f{h}=\f{g}/\mathrm{Span}\{\xi\}$ such that there are
stable forms $\omega$, $\psi_+$ and $\psi_-=\hat{\psi}_+$
satisfying
\begin{equation}\label{compatibility}d\omega=0,\hspace{1em} d\psi_-=0,\hspace{1em} \omega\wedge
d\eta+d\psi_+=0,\hspace{1em} \psi_-\wedge d\eta=0,\hspace{1em}
\omega\wedge\psi_-=0.\end{equation} For every Lie algebra $\f{g}$
in \eqref{list1} we shall find an element $\xi$ in the center such
that all the solutions of \eqref{compatibility} are either non
stable forms, or have associated metric $g_{\varphi}$ (with
$\varphi=\omega\wedge\eta+\psi_+$) not definite on the center,
yielding a contradiction to the fact that $\f{g}$ admits a
torsion-free $G^*_{2(2)}$-structure with definite center. Note
that if for example we reduce by $\xi=e_7$, the general form of
$\eta=-\varepsilon\xi^{\flat}$ is
$$\eta=e^7+\sum_{i=1}^6\gamma_ie^i,$$
for some constants $\gamma_i$ depending on $g_{\varphi}$.

\begin{itemize}

\item
$(0,0,12,0,0,13+24,15)\overset{\xi=e_7}{\rightarrow}(0,0,12,0,0,13+24)$:
For any compatible closed forms $\omega$ and $\psi_-$ on $\f{h}$,
imposing $\omega\wedge d\eta+d\psi_+=0$,and  $\psi_-\wedge
d\eta=0$ we obtain solutions such that $e_6$ is null with respect
to the corresponding metric $g_{\varphi}$ on $\f{g}$.

\item
$(0,0,12,0,24+13,14,46+24+15+23)\overset{\xi=e_7}{\rightarrow}(0,0,12,0,24+13,14)$:
For any closed forms $\omega$ and $\psi_-$ on $\f{h}$, equations
$\psi_-\wedge d\eta=0$, $\omega\wedge\psi_-=0$ and $\omega\wedge
d\eta+d\psi_+=0$ imply $\omega^3=0$.

\item
$(0,0,12,0,0,13,14+25)\overset{\xi=e_6}{\rightarrow}(0,0,12,0,0,14+25)$:
For any closed forms $\omega$ and $\psi_-$ on $\f{h}$, equation
$\omega\wedge d\eta+d\psi_+=0$ implies $\omega^3=0$.

\item $(0,0,12,0,13,24+23,25+34+16+15+3
\cdot{26})\overset{\xi=e_7}{\rightarrow}(0,0,12,0,13,24+23)$: For
any closed forms $\omega$ and $\psi_-$ on $\f{h}$, equation
$\omega\wedge d\eta+d\psi_+=0$ implies $\omega^3=0$.

\item
$(0,0,0,12,13,14+23,15)\overset{\xi=e_7}{\rightarrow}(0,0,0,12,13,14+23)$:
For any compatible closed forms $\omega$ and $\psi_-$ on $\f{h}$,
imposing $\omega\wedge d\eta+d\psi_+=0$,and  $\psi_-\wedge
d\eta=0$ we obtain solutions such that $e_6$ is null with respect
to the corresponding metric $g_{\varphi}$ on $\f{g}$.

\item
$(0,0,0,12,13,14,15)\overset{\xi=e_7}{\rightarrow}(0,0,0,12,13,14)$:
For any closed forms $\omega$ and $\psi_-$ on $\f{h}$, equations
$\psi_-\wedge d\eta=0$, $\omega\wedge\psi_-=0$ and $\omega\wedge
d\eta+d\psi_+=0$ imply $\omega^3=0$ or $\lambda(\psi_-)=0$.

\item
$(0,0,12,13,23,15+24,16+34)\overset{\xi=e_7}{\rightarrow}(0,0,12,13,23,15+24)$:
For any closed forms $\omega$ and $\psi_-$ on $\f{h}$, equations
$\psi_-\wedge d\eta=0$ and $\omega\wedge d\eta+d\psi_+=0$ imply
$\omega^3=0$.

\item
$(0,0,12,13,23,15+24,16+25+34)\overset{\xi=e_7}{\rightarrow}(0,0,12,13,23,15+24)$:
For any closed forms $\omega$ and $\psi_-$ on $\f{h}$, equations
$\psi_-\wedge d\eta=0$, $\omega\wedge\psi_-=0$ and $\omega\wedge
d\eta+d\psi_+=0$ imply $\omega^3=0$.
\end{itemize}
\end{proof}

\begin{theorem}\label{proposition clasification indecomposable}
The only indecomposable $7$-dimensional nilpotent Lie algebra
admitting a torsion-free $G_{2(2)}^*$-structure with definite
center is
$$\f{g}_1=(0,0,0,12,23,-13,2 \cdot {26}-2 \cdot{34}-2 \cdot{16}+ 2 \cdot  {25}).$$
Moreover, $\f{g}_1$  admits a $3$-parameter family of torsion-free
$G_{2(2)}^*$-structures  with holonomy equal to $G_{2(2)}^*$.
\end{theorem}

\begin{proof}
We first show that $\f{g}_1$ admits torsion-free
$G_{2(2)}^*$-structures with definite center. We
consider  a  new basis $\{e_1,\ldots,e_7\}$   of $\f{g}_1$  with respect to which $\f{g}_1$
has structure equations
$$(0,0,0,12,13,23,-2\cdot 25-2\cdot 34+2\cdot 15+2\cdot 26)$$
(this can be done by setting $e_5=-e_6'$, $e_6=e_5'$, and
$e_i=e_i'$ for $i\neq 5,6$, where by  $\{e_1',\ldots,e_7'\}$  we denote
the old basis  of  ${\frak g}_1$). Consider the reduction
$$\f{g}_1\rarrow \f{h}_1=\frac{\f{g}_1}{\mathrm{span}\{e_7\}}\cong (0,0,0,12,13,23).$$
Let $\varphi$ be a $G_{2(2)}^*$-structure on $\f{g}_1$. As before
we write
$$\varphi=\omega\wedge \eta+\psi_+, \qquad
*\varphi=\varepsilon\psi_-\wedge\eta-\frac{1}{2}\varepsilon\omega\wedge\omega.$$
We will now compute all torsion-free $G_{2(2)}^*$-structures with
definite center satisfying $\eta = e^7$ and
$g_{\varphi}(e_7,e_7)=-\varepsilon\in\{\pm 1\}$, by solving
equations \eqref{compatibility} for stable forms $\omega$,
$\psi_+$ and $\psi_-=\hat{\psi}_+$ on $\f{h}_1$. We will also
impose the normalization condition
$\psi_-\wedge\psi_+=\frac{2}{3}\omega^3.$ The closed $2$-forms  on
$\f{h}_1$ are given by
\begin{align}
\omega
&=r_1e^{12}+r_2e^{13}+r_3e^{14}+r_4e^{15}+r_5e^{23}+r_6e^{24}+r_7e^{26}+r_8e^{35}\nonumber\\
&+r_9e^{36}+r_{10}(e^{16}+e^{25})+r_{11}(e^{16}-e^{34}),
\end{align}
for some   parameters $r_1, \ldots, r_{11} \in \R$. The
non-degeneracy condition $\omega^3\neq 0$ is thus
\begin{equation}\label{mu}
r_3(r_7r_8-r_9r_{10})+r_{4}(r_{6}r_9+r_7r_{11})-(r_6 r_8 + r_{10}
r_{11})  (r_{10}+r_{11}) \neq 0.
\end{equation}
For the second equation of \eqref{compatibility}, it is also easy
to compute the space of closed $3$-forms. In this case imposing
also the fourth equation we obtain that $\psi_-$ must be of the
form
\begin{align}
\psi_- & =m_1(e^{125}-e^{234})+m_2(e^{126}+e^{134}-e^{125})+m_3(e^{135}-e^{136})\nonumber\\
& +m_4(e^{135}-e^{236})+m_5(e^{145}-e^{146}-e^{245})+m_6(e^{145}-e^{246})\nonumber\\
&+m_7e^{123}  +m_8e^{124}+m_9e^{235}+m_{10}e^{356},
\end{align}
for some parameters $m_1,\ldots,m_{10}\in\R$. The non-degeneracy
condition for $\psi_-$ is $\lambda(\psi_-)\neq 0$ (which is
equivalent to $\lambda(\psi_+)\neq 0$) where $\lambda$ is given by
\eqref{lambda stable}. To impose the third equation of
\eqref{compatibility} we have to compute $\psi_+$. Since
$\psi_-=J_{\psi_+}^*\psi_+$, where $J_{\psi_+}$ is given by
\eqref{J rho stable}, taking into account
$J_{\psi_+}^2=\varepsilon$ and \eqref{estructuras complejas} we
obtain $\psi_+=-J_{\psi_-}^*\psi_-$.

The two remaining conditions are then the vanishing of the forms
\begin{align}\label{compat1}&\omega\wedge\psi_-,\\
\label{compat2}& d\psi_++\omega\wedge d\eta.\end{align} Solving
these two equations and imposing the normalization condition, we
obtain after a quite long computation a parametrization of the
space of torsion-free $G_{2(2)}^*$ structures on $\f{g}_1$ such
that $g_{\varphi}(e_7,e_7)=-\varepsilon$. Such parame\-trization
is a subset of $\bb{R}^{11}\times\bb{R}^{10}\times\bb{Z}_2$ given
by
$$
\begin{array}{c}
r_3=0,\hspace{1em} r_6=0,\hspace{1em} r_8=0,\hspace{1em} r_9=0,\hspace{1em} r_7=r_4,\hspace{1em} r_{11}=-r_4,\hspace{1em}r_{10}=0,\\[4 pt] m_1 =-\frac{1}{3}\frac{r_2m_6+r_5m_5+2r_2m_5+2r_5m_6}{r_4}, \quad
m_2=-\frac{1}{3}\frac{2r_2m_6-r_5m_5+r_2m_5+r_5m_6}{r_4},\\[4pt]
m_3 =-\frac{r_4m_9+r_1m_{10}}{r_4},
\end{array}
$$
$$
\begin{array}{lcl}
m_7 &=&-\frac{1}{12 m_{10}^2r_4^2(m_5^2+m_6^2+m_5m_6)^2}
\left( 18 m_5m_9m_6^2m_8m_{10}^2r_4^2  +9 m_5^4m_4^2r_4^2m_{10}\right.\nonumber\\
&&+ 18 m_4m_5^2m_9m_6r_4^2m_{10} -18 m_4m_5m_{10}^2m_6^2m_8r_4^2 +9 m_9^2m_6^4r_4^2m_{10}\nonumber\\
&& + 18 m_4m_5^2m_9m_6^2r_4^2m_{10} +18 m_5^2m_6m_8m_{10}^2r_4^2m_9-12 m_5^2m_6^2r_1^2m_{10}^3\nonumber\\
&& +6 m_4m_6^3m_5r_4r_1m_{10}^2-6 m_5^3m_9m_6r_4r_1m_{10}^2+12 m_5m_6^2m_8m_{10}^3r_4r_1\nonumber\\
&& +8 m_{10}^2m_6^4r_5^2m_5+8 m_{10}^2m_6^3r_2^2m_5^2+8 m_{10}^2m_5^4r_2^2m_6 +8 m_{10}^2m_6^5r_2r_5\nonumber\\
&& +12 m_{10}^2m_5^3r_2^2m_6^2+4 m_{10}^2m_5^5r_2^2+4 m_{10}^2m_6^5r_5^2+8 m_{10}^2m_5^5r_5r_2\nonumber\\
&& +12 m_{10}^2m_6^3r_5^2m_5^2-3 m_8^2m_{10}^2m_6^2r_4^2-3 m_8^2m_{10}^2m_5^2r_4^2 +8 m_{10}^2m_5^3r_5^2m_6^2\nonumber\\
&&+4 m_{10}^2m_5^4r_5^2m_6+18 m_4m_6^3m_5r_4^2m_9m_{10}  -18 m_4m_5^2m_{10}^2m_6m_3r_4^2  \nonumber\\
&&-9 m_5m_6^3r_1^2m_{10}^3+4 m_{10}^2m_{6}^4r_2^2m_5+6 m_5^3m_8m_{10}^3r_4r_1-3 m_6^4r_1^2m_{10}^3\nonumber\\
&&-9 m_5^3m_6r_1^2m_{10}^3+6 m_5m_{10}^3m_6^3r_4r_1+6 m_9m_6^4r_4r_1m_{10}^2 -3 m_5^4r_1^2m_{10}^3\nonumber\\
&& -6 m_4m_5^4r_4r_1m_{10}^2+9 m_5m_9^2m_6^3r_4^2m_{10}+6 m_8m_{10}^2m_6^3r_4^2m_9\nonumber\\
&& +9 m_5^2m_9^2m_6^2r_4^2m_{10}-6 m_4m_5^3m_{10}^2m_8r_4^2-12 m_4m_6^3m_8m_{10}^2r_4^2\nonumber\\
&& -3 m_5m_8^2m_{10}^3m_6r_4^2+40 m_{10}^2m_6^3r_5m_5^2r_2+24 m_{10}^2m_5^4r_2m_6r_5\nonumber\\
&& +40 m_{10}^2m_5^3r_2m_6^2r_5+9 m_5^3m_4^2m_6r_4^2m_{10}+9 m_5^2m_4^2m_6^2r_4^2m_{10}\nonumber\\
&& +12 m_5^2m_8m_{10}^2r_4^2m_9+24
m_{10}^2m_6^4r_2r_5m_5+3\phi^3r_4^3 ),
\end{array}
$$

$$m_{10}=2\varepsilon\frac{r_4^4}{m_5^2+m_6^2+m_5m_6},\qquad\phi=-2r_4^3,$$
and non-degeneracy conditions $r_4\neq 0$, $m_5^2+m_6^2+m_5m_6\neq
0$.  Note that this space has (at least) four connected components
due to the value of $(\varepsilon,\mathrm{sign}(r_4))$. The space
of free parameters
\begin{equation}\label{free paramet case full hol
1}\left(r_1,r_2,r_4,r_5,m_4,m_5,m_6,m_8,m_9,\varepsilon\right)\end{equation}
is then an open set of $\bb{R}^4\times\bb{R}^5\times\bb{Z}_2$
given by the non-degeneracy condition
$$r_4\neq 0,\qquad m_5^2+m_6^2+m_5m_6\neq 0.$$

The rest of the proof is devoted to exhibit a subfamily of
torsion-free $G_{2(2)}^*$ structures on $\f{g}_1$ with full
holonomy $G_{2(2)}^*$. Since the non-degeneracy conditions only
involve $(r_4,m_5,m_6)$, we choose the values
\begin{equation}\label{choice of parameters}r_1=0,\hspace{1em} r_2=0,\hspace{1em}r_5=0,\hspace{1em}m_4=0,\hspace{1em}m_8=0,\hspace{1em}m_9=0.\end{equation}
Recall that since the pair $(\omega,\psi_+)$ is normalized, the
inner product  $g_{\varphi}$ on $\f{g}_1$ is obtained as
$g_{\varphi}=h-\varepsilon\eta\otimes\eta,$ where $h$ is the inner
product  associated to $(\omega,\psi_+)$. Evaluating at
\eqref{choice of parameters} we obtain for $g_{\varphi}$ the
matrix representation
\begin{equation*}
{\small \begin{pmatrix}
-\frac{m_6}{2} & \frac{m_5+m_6}{2} & 0 & 0 & 0 & 0 & 0\\
\frac{m_5+m_6}{2} & -\frac{ m_5}{2} & 0 & 0 & 0 & 0 & 0\\
0 & 0 & \frac{\varepsilon r_4^4}{(m_5,m_6)} & 0 & 0 & 0 & 0\\
0 & 0 & 0 & -\frac{(m_5,m_6)}{r_4^2} & 0 & 0 & 0\\
0 & 0 & 0 & 0 & -\frac{2\varepsilon r_4^2m_5}{(m_5,m_6)} & -\frac{2\varepsilon r_4^2m_6}{(m_5,m_6)} & 0\\
0 & 0 & 0 & 0 & -\frac{2\varepsilon r_4^2m_6}{(m_5,m_6)} & \frac{2\varepsilon r_4^2(m_5+m_6)}{(m_5,m_6)} & 0\\
0 & 0 & 0 & 0 & 0 & 0 & -\varepsilon
\end{pmatrix},}
\end{equation*}
where $(m_5,m_6)$ stands for $m_5^2+m_6^2+m_5m_6$.

In order to show that this family of metrics has full holonomy
$G_{2(2)}^*$ we show that the Lie algebra spanned by the curvature
endomorphisms $R_{XY}:\f{g}_1\to\f{g}_1$, $X,Y\in\f{g}_1$, and the
endomorphisms $(\nabla_ZR)_{XY}:\f{g}_1\to\f{g}_1$,
$Z,X,Y\in\f{g}_1$, has dimension equal to $14$. As it is well
known (see \cite{Ambrose-Singer}), the holonomy algebra $\f{hol}$
of $g_{\varphi}$ contains this Lie algebra. Therefore since
$\f{hol}\subset\f{g}^*_{2(2)}$ and $\mathrm{dim}(\f{g}_{2(2)})=14$
we have that $\f{hol}=\f{g}^*_{2(2)}$, so that $g_{\varphi}$ has
full holonomy $G^*_{2(2)}$. A set of linearly
independent endomorphisms spanning $\f{hol}$ can be found in the
Appendix at the end of the manuscript.
\end{proof}

The three parameter family of metrics with full holonomy
constructed on $\f{g}_1$ provides a family of left-invariant
metrics with full holonomy $G^*_{2(2)}$ on the simply-connected
nilpotent Lie group $G$ associated to $\f{g}_1$. It is worth
noting the difference with the Riemannian counterpart, where any
homogeneous Ricci-flat metric must be flat. Moreover, since the
structure equations of $\f{g}_1$ are rational, $G$ admits a
co-compact lattice $\Gamma$ (see \cite{Malcev}). Therefore the
family of left-invariant metrics on $G$ induce a family of metrics
with full holonomy $G^*_{2(2)}$ on the compact nilmanifold
$\Gamma\setminus G$.

\begin{remark}
So far, the authors have not found any value of the parameters
\eqref{free paramet case full hol 1} such that $g_{\varphi}$ has
not holonomy equal to $G_{2(2)}^*$.
\end{remark}

\begin{remark}
Note that we have parameterized all  torsion-free
$G_{2(2)}^*$-struc\-tures with definite center on $\f{g}_1$ but we
have not studied their isomorphism classes.
There are choices for the parameters that give isometric
structures and choices that give non-isometric structures. For
example setting \eqref{choice of parameters}, the choices
$\{r_4=-1,m_5=1,m_6=-1\}$ and $\{r_4=-1,m_5=-1,m_6=1\}$ for
$\varepsilon=\pm 1$ give isometric metrics.
%
On the other hand, the choices $\{r_4=-1,m_5=1,m_6=0\}$ and
$\{r_4=-2,m_5=1,m_6=0\}$ give non-isometric metrics.
\end{remark}

As it has been pointed out, the condition that a nilpotent Lie
algebra  has definite center with respect to $g_{\varphi}$ is
rather strong. New examples of torsion-free
$G^*_{2(2)}$-structures can be obtained relaxing this condition
and just asking the existence of a non-null element in the center
of the Lie algebra. This is done for instance in the proof of
Lemma \ref{lemma eliminating} for the Lie algebras
$(0,0,12,0,0,13+24,15)$ and $(0,0,0,12,13,14+23,15)$.

\begin{example} Consider the   Lie algebra  $\f{g}= (0,0,0,12,13,15,14+23)$ and the reduction
$$(0,0,0,12,13,15,14+23)\overset{e_7}{\to} \f{h}=(0,0,0,12,13,15).$$
We construct a pair $(\omega,\psi_-)$ of stable forms on $\f{h}$
satisfying \begin{equation}\label{equations
example}d\omega=0,\hspace{1em} d\psi_-=0,\hspace{1em} \omega\wedge
d\eta+d\psi_+=0,\hspace{1em} \psi_-\wedge d\eta=0,\hspace{1em}
\omega\wedge\psi_-=0,\end{equation} and the normalization
condition
$$\psi_-\wedge\psi_+=\frac{2}{3}\omega^3.$$
For the sake of simplicity we shall restrict ourselves to the case
$\eta=e^7$. The closed $2$-forms on $\f{h}$ are given by
\begin{align*}
\omega
&=r_{12}e^{12}+r_{13}e^{13}+r_{14}e^{14}+r_{15}e^{15}+r_{16}e^{16}+r_{23}e^{23}+r_{24}e^{24}\nonumber\\&
+r_{34}(e^{25}+e^{34})+r_{35}e^{35},
\end{align*}
for some   parameters $r_{ij} \in \R$. The non-degeneracy
condition $\omega^3\neq 0$ is thus
\begin{equation*}
-r_{16}r_{24}r_{35}+r_{16}r_{34}^2\neq 0.
\end{equation*}
A closed $3$-form $\psi_-$ is of the form
\begin{align*}
\psi_- & =  m_{123}e^{123}+m_{124}e^{124}+m_{125}e^{125}+m_{126}e^{126}\nonumber\\
& +m_{134}e^{134}+m_{135}e^{135}+m_{136}e^{136}+m_{145}e^{145}\nonumber\\
&
+m_{146}e^{146}+m_{156}e^{156}+m_{234}e^{234}+m_{235}e^{235}\nonumber\\
& +m_{345}(e^{236}+e^{345})+m_{356}e^{356},
\end{align*}
for some parameters $m_{ijk}\in\R$. The non-degeneracy condition
for $\psi_-$ is $\lambda(\psi_-)\neq 0$. Solving \eqref{equations
example} together with the normalization condition we obtain
$$r_{14}=-r_{23},\hspace{1em}r_{34}=0,\hspace{1em}r_{16}=r_{35},$$
$$m_{145}+m_{235}=0,\hspace{1em}m_{345}+m_{146}=0,\hspace{1em}m_{356}=m_{156}=0,$$
\begin{align*}
m_{124}&=-\frac{1}{r_{35}}\left(2m_{235}r_{23}+m_{135}r_{24}+m_{234}r_{15}-m_{345}r_{12}\right)\\
m_{126}&=\frac{1}{r_{35}}\left(m_{235}r_{35}-m_{345}r_{15}\right)\\
m_{136}&=-\frac{1}{r_{24}}\left(m_{234}r_{35}+2m_{345}r_{23}\right)\\
m_{125}&=\frac{1}{8m_{345}^6r_{35}^2}\left(6m_{345}^5m_{235}^2r_{35}^2-4m_{345}^6m_{235}r_{35}r_{15}\right.\\
&\left.-2m_{345}^7r_{15}^2+r_{35}^3\phi^3 \right)
\end{align*}
$$r_{24}=-\varepsilon\frac{|m_{345}|^3}{|r_{35}|^3},\qquad \phi=2\varepsilon\frac{|m_{345}|^3}{|r_{35}|},$$
with non-degeneracy condition $r_{35}\neq 0$ and $m_{345}\neq 0$.
Note that this parametrization has at least eight connected
components due to the value of the triple
$(\varepsilon,\mathrm{sign}(m_{345}),\mathrm{sign}(r_{35}))$. The
space of free parameters
\begin{equation}\label{free paramet case full
hol
2}(r_{12},r_{13},r_{15},r_{23},r_{35},m_{123},m_{134},m_{135},m_{234},m_{235},m_{345},\varepsilon)\end{equation}
is thus an open set of $\R^5\times\R^6\times\mathbb{Z}_2$ given by
equations $m_{345}\neq 0$ and $r_{35}\neq 0$.

Since the degeneracy condition only involves $r_{35}$ and
$m_{345}$, we select the subfamily of $G^*_{2(2)}$-structures
given by the choice
$$r_{12}=0,r_{13}=0,r_{15}=0,r_{23}=0,m_{123}=0,m_{134}=0,m_{135}=0,$$
$$m_{234}=0,m_{235}=0.$$
For the sake of simplicity we also suppose $r_{35}>0$ and
$m_{345}>0$ (the other cases are analogous). Evaluating at these
values, we obtain
\begin{align*}
\varphi & = \varepsilon
\frac{m_{345}^2}{r_{35}}\left(-e^{145}+e^{126}+e^{235}\right)
-\varepsilon
\frac{m_{345}^3}{r_{35}^3}e^{247}+r_{35}\left(e^{346}+e^{167}+e^{357}\right),
\end{align*}
whence $g_{\varphi}$ has matrix representation
$$\begin{pmatrix}
0 & 0 & 0 & 0 & m_{345} & 0\\
0 & \frac{\varepsilon m_{345}^4}{r_{35}^4} & 0 & 0 & 0 & 0 & 0\\
0 & 0 & 0 & 0 & 0 & \frac{\varepsilon r_{35}^2}{m_{345}} & 0\\
0 & 0 & 0 & -\frac{m_{345}^2}{r_{35}^2} & 0 & 0 & 0\\
m_{345} & 0 & 0 & 0 & 0 & 0 & 0\\
0 & 0 & \frac{\varepsilon r_{35}^2}{m_{345}} & 0 & 0 & 0 & 0\\
0 & 0 & 0 & 0 & 0 & 0 & -\varepsilon
\end{pmatrix}.$$
Note that $e_6$ is null as expected. A straightforward computation
shows that the space spanned by the curvature endomorphisms of
$g_{\varphi}$ and its first derivatives has dimension $6$. In
addition, the second order derivatives of the curvature are
linearly dependent with the curvature and its first derivatives.
This implies that the holonomy of $g_{\varphi}$ is
$6$-dimensional. So far, the authors have not found any value of
the parameters \eqref{free paramet case full hol 2} such that
$g_{\varphi}$ has not $6$-dimensional holonomy.
\end{example}

\begin{remark}
As in the previous example, all the torsion-free
$G^*_{2(2)}$-struc\-tures on the Lie algebra
$(0,0,0,12,13,14+23,15)$ obtained by the authors have
$6$-dimensional holonomy. In this case, in order to obtain
solutions to \eqref{compatibility} one must consider
$\eta=e^7+\sum_{i=1}^6\gamma_ie^i$, with $\gamma_6\neq 0$.
\end{remark}

\begin{remark} The reduction procedure shown at the beginning of
this section gives a simple method for constructing torsion-free
$G^*_{2(2)}$-structures on decomposable $7$-dimen\-sional Lie
algebras of the form $\f{g}=\f{h}\oplus\R$. Recall that by
\cite{Gong} the only decomposable $7$-dimensional nilpotent Lie
algebra which is not of this form is $(0,0,0,0,12,34,36)$.

Let $\f{g}=\f{h}\oplus\R$ be a $7$-dimensional nilpotent Lie
algebra. Let $\varphi$ be a torsion-free $G^*_{2(2)}$-structure on
$\f{g}$ such that the summand $\R$ is not null and orthogonal to
$\f{h}$  with respect to $g_{\varphi}$. The reduction
$\f{g}\to\f{h}$ has curvature form $d\eta=0$, so that writing
again \eqref{varphi} and \eqref{starvarphi}, equations
\eqref{additional equations} transform into $d\psi_+=0$.
Therefore, constructing a torsion-free $G_{2(2)}^*$-structure on
$\f{g}$ with non-null summand $\R$ orthogonal to $\f{h}$ is
equivalent to construct a torsion-free almost
$\varepsilon$-special Hermitian structure on $\f{h}$, i.e., a
compatible pair $(\omega,\psi_+)$ with $d\omega=0$, $d\psi_+=0$,
$d\psi_-=0$  $($\cite{CCLS}$)$. Note that with this procedure the
holonomy of $g_{\varphi}$ will be equal to the holonomy of
$g_{\varphi|\f{h}}$, and in particular it will be contained in
$SL(3,\R)$ or $SU(2,1)$.
\end{remark}

\section*{Appendix}

We denote by $e_i^j$ the endomorphism  $e_i\otimes e^j$ of
$\f{g}_1$ and   $(m_5,m_6)=m_5^2+m_6^2+m_5m_6$, where $\{e_1, \ldots, e_7\}$ is the basis such that $\f{g}_1$ has structure equations $(0,0,0,12,13,23,-2\cdot 25-2\cdot 34+2\cdot 15+2\cdot 26)$.   By direct
calculation, a set of  linearly independent endomorphisms of $\f{g}_1$ spanning
$\f{hol}$ is
\begin{align*}
R_{e_1e_2} &
=-\frac{3(m_5+m_6)}{2r_4^2}e_1^1-\frac{3m_6}{2r_4^2}e_1^2+\frac{3m_5}{2r_4^2}e_2^1+\frac{3(m_5+m_6)}{2r_4^2}e_2^2-\frac{1}{2}e_3^7\\
&+\frac{m_6}{r_4^2}e_5^5-\frac{m_5}{r_4^2}e_5^6-\frac{m_5+m_6}{r_4^2}e_6^5-\frac{m_6}{r_4^2}e_6^6-\frac{(m_5,m_6)}{2r_4^2}e_7^3,\\
R_{e_1e_3} &
=-\frac{3m_5}{2r_4^2}e_1^3+\frac{1}{2}e_1^7-\frac{2m_6}{2r_4^2}e_2^3-\frac{1}{2}e_2^7+\frac{3\e
r_4}{(m_5,m_6)}e_3^1+\frac{3\e r_4^2}{(m_5,m_6)}e_3^2\\
&+\frac{m_5+m_6}{r_4^2}e_4^5+\frac{m_6}{r_4^2}e_4^6-\frac{2\e
r_4^2}{(m_5,m_6)}e_5^4-\frac{\e m_6}{(m_5,m_6)}e_7^1+\frac{\e
m_5}{(m_5,m_6)}e_7^2,\\
R_{e_1e_4} &
=\frac{m_6}{2r_4^2}e_1^4-\frac{m_5+m_6}{2r_4^2}e_2^4-\frac{m_5+m_6}{2r_4^2}e_3^5-\frac{m_6}{2r_4^2}e_3^6-\frac{(m_5,m_6)}{r_4^4}e_4^1\\
&
-\frac{(m_5,m_6)}{r_4^4}e_5^3-\frac{m_5}{r_4^2}e_5^7-\frac{m_6}{r_4^2}e_6^7+\frac{(m_5,m_6)}{2r_4^4}e_7^5,\\
R_{e_1e_5} &
=\frac{4m_5+3m_6}{2r_4^2}e_1^5+\frac{3m_6}{2r_4^2}e_1^6-\frac{3m_5}{2r_4^2}e_2^5-\frac{2m_5+3m_6}{2r_4^2}e_2^6-\frac{2\e
r_4^2}{(m_5,m_6)}e_3^4\\
&-\frac{3(m_5,m_6)}{r_4^4}e_4^3+\frac{m_5}{r_4^2}e_4^7+\frac{2\e(m_5-3m_6)}{(m_5,m_6)}e_5^1+\frac{8\e
m_5}{(m_5,m_6)}e_5^2\\ &+\frac{2\e(2m_5+3m_6)}{(m_5,m_6)}e_6^1+\frac{6\e m_6}{(m_5,m_6)}e_6^2-\frac{\e m_5}{(m_5,m_6)}e_7^4,\\
R_{e_1e_6} & =
\frac{m_6}{2r_4^2}e_1^5+\frac{m_6}{2r_4^2}e_2^6+\frac{m_6}{r_4^2}e_4^7+\frac{2\e
m_6}{(m_5,m_6)}e_5^1+\frac{2\e m_6}{(m_5,m_6)}e_5^2-\frac{2\e
m_6}{(m_5,m_6)}e_6^1\\ &-\frac{\e m_6}{(m_5,m_6)}e_7^4,\\
 R_{e_1e_7}
&
=-\frac{(m_5,m_6)}{2r_4^4}e_1^3-\frac{m_5+m_6}{2r_4^2}e_1^7+\frac{m_5}{2r_4^2}e_2^7+\frac{\e
m_5}{(m_5,m_6)}e_3^1+\frac{\e(m_5+m_6)}{(m_5,m_6)}e_3^2\\
&+\frac{(m_5,m_6)}{r_4^4}e_4^5-\frac{2\e
m_5}{(m_5,m_6)}e_5^4-\frac{2\e
m_6}{(m_5,m_6)}e_6^4-\frac{\e}{r_4^2}e_7^2,
\end{align*}
\begin{align*}
 R_{e_2e_3} &
=-\frac{3m_5}{2r_4^2}e_1^3+\frac{3(m_5+m_6)}{2r_4^2}e_2^3+\frac{1}{2}e_2^7-\frac{3\e
r_4}{(m_5,m_6)}e_3^1-\frac{m_5}{r_4^2}e_4^5-\frac{m_5+m_6}{r_4^2}e_4^6\\
&+\frac{2\e r_4^2}{(m_5,m_6)}e_5^4-\frac{2\e
r_4^2}{(m_5,m_6)}e_6^4+\frac{\e(m_5+m_6)}{(m_5,m_6)}e_7^1+\frac{\e
m_6}{(m_5,m_6)}e_7^2,\\
 R_{e_2e_4} &
=-\frac{m_5+m_6}{2r_4^2}e_1^4+\frac{m_5}{2r_4^2}e_2^4+\frac{m_5}{2r_4^2}e_3^5+\frac{m_5+m_6}{2r_4^2}e_3^6-\frac{(m_5,m_6)}{r_4^4}e_4^2
\\
&+\frac{(m_5,m_6)}{r_4^4}e_5^3-\frac{m_6}{r_4^2}e_5^7-\frac{(m_5,m_6)}{r_4^4}e_6^3+\frac{m_5+m_6}{r_4^2}e_6^7+\frac{(m_5,m_6)}{2r_4^4}e_7^6,\\
R_{e_2e_7} &
=\frac{(m_5,m_6)}{2r_4^4}e_1^3+\frac{m_5}{2r_4^2}e_1^7-\frac{(m_5,m_6)}{2r_4^4}e_2^3+\frac{m_6}{2r_4^2}e_2^7+\frac{\e
m_6}{(m_5,m_6)}e_3^1\\ &-\frac{\e
m_5}{(m_5,m_6)}e_3^2+\frac{(m_5,m_6)}{r_4^4}e_4^6-\frac{2\e
m_6}{(m_5,m_6)}e_5^4+\frac{2\e(m_5+m_6)}{(m_5,m_6)}e_6^4+\frac{\e}{r_4^2}e_7^1+\frac{\e}{r_4^2}e_7^2,\\
R_{e_3e_7} & = \frac{\e(m_5+m_6)}{(m_5,m_6)}e_1^1+\frac{\e
m_6}{(m_5,m_6)}e_1^2-\frac{\e
m_5}{(m_5,m_6)}e_2^1-\frac{\e(m_5+m_6)}{(m_5,m_6)}e_2^2-\frac{\e
r_4^2}{(m_5,m_6)}e_3^7\\ &-\frac{2\e
m_6}{(m_5,m_6)}e_5^5+\frac{2\e
m_5}{(m_5,m_6)}e_5^6+\frac{2\e(m_5+m_6)}{(m_5,m_6)}e_6^5+\frac{2\e
m_6}{(m_5,m_6)}e_6^6-\frac{\e}{r_4^2}e_7^3,
\end{align*}
\begin{align*}
(\nabla_{e_1}R)_{e_1e_2} &
=\frac{m_6}{r_4^2}e_1^4-\frac{m_5+m_6}{r_4^2}e_2^4-\frac{m_5}{2r_4^2}e_3^5-\frac{m_5+m_6}{2r_4^2}e_3^6-\frac{2(m_5,m_6)}{r_4^4}e_4^1\\
&-\frac{(m_5,m_6)}{r_4^4}e_5^3-\frac{m_5+m_6}{r_4^2}e_5^7+\frac{(m_5,m_6)}{r_4^4}e_6^3+\frac{m_5}{r_4^2}e_6^7+\frac{(m_5,m_6)}{2r_4^4}e_7^5\\
&+\frac{(m_5,m_6)}{2r_4^4}e_7^6,\\
(\nabla_{e_1}R)_{e_1e_3} &
=\frac{2m_5+m_6}{r_4^2}e_1^5+\frac{m_6}{r_4^2}e_1^6-\frac{m_5-2m_6}{2r_4^2}e_2^5-\frac{m_5+m_6}{2r_4^2}e_2^6-\frac{4\e
r_4^4}{(m_5,m_6)}e_3^4\\
&-\frac{4(m_5,m_6)}{r_4^4}e_4^3-\frac{m_6}{r_4^2}e_4^7+\frac{2\e(3m_5-m_6)}{(m_5,m_6)}e_5^1
+\frac{2\e(4m_5+m_6)}{(m_5,m_6)}e_5^2\\
&+\frac{2\e(m_5+3m_6)}{(m_5,m_6)}e_6^1+\frac{6\e
m_6}{(m_5,m_6)}e_6^2+\frac{\e m_6}{(m_5,m_6)}e_7^4,\\
(\nabla_{e_1}R)_{e_1e_7} &
=-\frac{2m_5m_6+m_5^2+m_6^2}{r_4^4}e_1^5-\frac{m_6(m_5+m_6)}{r_4^4}e_1^6+\frac{3m_5m_6+3m_5^2+m_6^2}{2r_4^4}e_2^5\\
&+\frac{3m_5m_6+m_5^2+m_6^2}{2r_4^4}e_2^6+\frac{\e
m_6}{(m_5,m_6)}e_3^4+\frac{m_6(m_5,m_6)}{r_4^6}e_4^3-\frac{2(m_5,m_6)}{r_4^4}e_4^7\\
&+\frac{2\e(2m_5m_6+m_5^2+m_6^2)}{r_4^2(m_5,m_6)}e_5^1-\frac{2\e(m_5m_6+2m_5^2+m_6^2)}{r_4^2(m_5,m_6)}e_5^2\\
&-\frac{2\e m_5(m_5+m_6)}{r_4^2(m_5,m_6)}e_6^1-\frac{2\e
m_5m_6}{r_4^2(m_5,m_6)}e_6^2+\frac{2\e}{r_4^2}e_7^4,
\end{align*}
\begin{align*}
(\nabla_{e_2}R)_{e_1e_2} &
=-\frac{m_5+m_6}{r_4^2}e_1^4+\frac{m_5}{r_4^2}e_2^4-\frac{m_6}{2r_4^2}e_3^5+\frac{m_5}{2r_4^2}e_3^6-\frac{2(m_5,m_6)}{r_4^4}e_4^2\\
&+\frac{m_5}{r_4^2}e_5^7
-\frac{(m_5,m_6)}{r_4^4}e_6^3+\frac{m_6}{r_4^2}e_6^7-\frac{(m_5,m_6)}{2r_4^4}e_7^5.
\end{align*}

%


\begin{thebibliography}{99}

\bibitem{Ambrose-Singer} W. Ambrose, I.M. Singer, \emph{A Theorem on Holonomy}, Trans. Amer. Math. Soc. {\bf 75} (3) 1953,
428--443.

\bibitem{Berger} M. Berger, \emph{Sur les groupes d'holonomie des vari\'et\'es \'a  connection affine et des vari\'et\'es riemanniennes}, Bull. Soc. Math. France {\bf 83} (1955), 279--330.

\bibitem{Bryant} R.L. Bryant, \emph{Metrics with exceptional holonomy},
Ann. Math. {\bf 126} (1987), 525--576.

\bibitem{Bryant2} R.L. Bryant, S.M. Salamon, \emph{On the Construction of some Complete Metrics with Exceptional
Holonomy}, Duke Math. Journ. {\bf 58} (1989), 829--850.


\bibitem{CF} D. Conti, M. Fern\'{a}ndez, \emph{Nilmanifolds with a calibrated
$G_2$-structure}, Differ. Geom.  Appl.  {\bf 29} (4) (2011),
493--506.





\bibitem{CCLS} V. Cortes, L. Schafer, T. Leistner, F. Schulte-Hengesbach, \emph{Half flat structures and special holonomy},  Proc.  London Math. Soc.  (3) {\bf 102} (2011)
113-158.


\bibitem{Freibert}M. Freibert,   \emph{Calibrated  and parallel  structures on almost abelian Lie algebras}, arXiv:1307.2542.

\bibitem{GW} C. R. Graham,  T. Wilse, \emph{Parallel tractor
extension and ambient metrics of holonomy split $G_2$}, J.
Differential Geom. {\bf 92} (3) (2012), 463--505.

\bibitem{Gray} A. Gray, \emph{Vector cross products on manifolds}, Trans. Amer. Math. Soc. {\bf 141}
(1969), 463--504.

\bibitem{Gong} M.P. Gong, \emph{Classification of nilpotent Lie algebras of
dimension 7 (over algebraically closed fields and R)}, Ph. D.
Thesis, University of Waterloo, Ontario, Canada, 1998.


%


\bibitem{Kath} I. Kath, \emph{$G^*_{2(2)}$-Structures on pseudo-Riemannian manifolds}, J. Geom. Phys. {\bf 27} (1998),
155--177.

\bibitem{Kath2} I. Kath, \emph{Indefinite symmetric spaces with $G^*_{2(2)}$-structure}, to appear in J. London Math. Soc.

\bibitem{Leistner-Nurowski} T. Leistner, P. Nurowski, \emph{Ambient metrics with exectional
holonomy}, Ann. Sc. Norm. Super. Pisa Cl. Sci. (5)  {\bf 11}
(2012), 407--436.


\bibitem{Malcev} A.I. Malcev, \emph{On a class of homogeneous spaces}, reprinted in Amer. Math. Soc.
Translations Series 1, 9 (1962), 276--307.

\bibitem{Nomizu}  K. Nomizu, \emph{On the cohomology of compact homogeneous spaces of nilpotent
Lie groups},  Ann. of Math. (2), {\bf 59} (1954), 531--538.


\bibitem{Sullivan} D. Sullivan, Differential Forms and the Topology of Manifolds, in Manifolds
Tokyo 1973, ed. A. Hattori, University of Tokyo Press 1975.

\bibitem{Willse} T. Willse, \emph{Highly symmetric 2-plane
fields on 5-manifolds and 5-dimensional Heisenberg group
holonomy}, Differ. Geom. Appl. {\bf 33} (2014), 81--111.



\end{thebibliography}
\end{document}